\documentclass[11pt,a4paper]{article}
\usepackage[top=3.5cm, bottom=3.5cm, left=3cm, right=3cm]{geometry}
\usepackage[latin1]{inputenc}
\usepackage{csquotes}
\usepackage{amsmath}
\usepackage{amsthm}
\usepackage{amsfonts}
\usepackage{amssymb}
\usepackage{graphics}
\usepackage{float}
\usepackage[hang,flushmargin]{footmisc} 

\usepackage{amsmath,amsthm,amssymb,graphicx, multicol, array, nccmath}
\usepackage{enumerate}
\usepackage{enumitem}
\usepackage{xcolor}
\usepackage{multicol}

\usepackage[pagebackref,bookmarks,colorlinks,breaklinks,linktoc=page]{hyperref}
\hypersetup{linkcolor=blue,citecolor=blue,filecolor=blue,urlcolor=blue} 

\usepackage{tabto}

\numberwithin{equation}{section}

\usepackage{tikz}
\usepackage{pgfplots}

\usetikzlibrary{matrix}

\usepackage{mathrsfs}

\DeclareMathOperator{\rad}{rad}

\newtheorem{thm}{Theorem}[section]
\newtheorem{lem}{Lemma}[section]

\newtheorem{conj}{Conjecture}[section]

\newtheorem{cor}{Corollary}[section]
\newtheorem{dfn}{Definition}[section]

\newtheorem{rmk}{Remark}[section]

\newcommand{\N}{\mathbb{N}}

\newcommand{\R}{\mathbb{R}}

\usepackage{fancyhdr}
\usepackage{titletoc}
\let\LaTeXStandardTableOfContents\tableofcontents

\renewcommand{\tableofcontents}{%
	\begingroup%
	\renewcommand{\bfseries}{\relax}%
	\LaTeXStandardTableOfContents%
	\endgroup%
}%

\usepackage{scalerel}[2016/12/29]

\title{Note on the Exceptional Set in the ABC Conjecture}
\date{}
\author{N. A. Carella}

\begin{document}
	\maketitle
	
\begin{abstract}
Fix $\varepsilon>0$, let $x>1$ be a large real number and let $\rad(n)=\prod_{p\mid n}p$ be the radical of an integer $n\geq1$. A triple $(a,b,c)$, with $a+b=c$ and $\gcd(a,b,c)=1$, such that $c>c_{\varepsilon}(\rad(abc))^{1+\varepsilon}$, is called exceptional triple. Recent works have proved that the cardinality $\#\mathscr{E}(x)$ of set $\mathscr{E}$ of exceptional triples satisfies $\#\mathscr{E}(x)=O(x^{33/50})$. This note proves that the cardinality of the set $\mathscr{E}(x)$ of exceptional triples $(a,b,c)$ is an infinite set unconditionally for any $\varepsilon<2/35$. 
		\let\thefootnote\relax\footnote{ \today \date{} 
			}
\end{abstract}
\tableofcontents

\section{Introduction }\label{S1001}\hypertarget{S1001}
The symbol $\N=\{1,2,3,\ldots\}$ denotes the set of natural numbers. Let $\varepsilon>0$ be a fixed small real number, let $x>1$ be a large real number and let $\rad(n)=\prod_{p\mid n}p$ be the radical of an integer $n\geq1$. The integer solutions $(a,b,c)\in \N\times\N\times \N$ of the equation $X+Y=Z$, with $\gcd(a,b,c)=1$, such that  
\begin{equation}
	c> c_{\varepsilon}\rad(abc)^{1+\varepsilon},
\end{equation}
are called exceptional triples. The set of exceptional triples is defined by
\begin{eqnarray}
	\mathscr{E}(x)&=&\left\{\,(a,b,c)\in\N^3:\gcd(a,b,c)=1; a+b=c;\right .\\[.2cm]
	&&\hskip 2.5 in\left . c\in (x,x+h]; \text{ and }	c> c_{\varepsilon}\rad(abc)^{1+\varepsilon}\,\right\}\nonumber,
\end{eqnarray}
where $h=o(x)$. The exceptional triples are closely related to counterexamples of the $abc$ conjecture. 
	\begin{conj}\label{conj1001.101}\hypertarget{conj1001.101}{\normalfont(The \textit{abc} conjecture)}   Given a fixed real number $\varepsilon > 0$, there exists a constant $c_{\varepsilon} $ 
	such that if $a + b = c$ and $\gcd(a, b, c) = 1$, then 
	\begin{equation}\label{eq1001.101.d}
		c\leq c_{\varepsilon}\rad(abc)^{1+\varepsilon}.
	\end{equation}
\end{conj}
There are several formulations of the $abc$ conjecture over the integers, as well as generalizations to number fields and function fields, see \cite{GT2002}, \cite{LS1990} and the vast literature. 	\\

It was recently proved that the cardinality of set $\mathscr{E}$ of exceptional triples satisfies $\#\mathscr{E}(x)=O(x^{33/50})$, see \cite{BT2024} and \cite{LJ2025}. This note provides a qualitative result about the exceptional set.
\begin{thm}\label{thm5757SI.940}\hypertarget{thm5757SI.940} For any small real number $\varepsilon>0$ and large $x\geq x_0$, the following hold.
	\begin{enumerate}[font=\normalfont, label=(\roman*)]
	\item$\displaystyle \mathscr{E}(x)\ne\varnothing .$ 	
	\item$\displaystyle \#\mathscr{E}(x)\to \infty$ as $x\to \infty$. 
\end{enumerate}	
In particular, the union of exceptional sets 	
\begin{equation}
\mathscr{E}=\bigcup_{x\geq x_0}	\mathscr{E}(x)
\end{equation}
is an infinite subset of $ \N\times\N\times \N$ and in the limit, the cardinality   	
\begin{equation}
\#	\mathscr{E}=\lim_{x\geq x_0}\#	\mathscr{E}(x)
\end{equation}
of the exceptional set $\mathscr{E}$ is unbounded.
\end{thm}
The proof is unconditional, and it is derived from the theory of smooth integers in short intervals. The fundamental background is covered in \hyperlink{S5555SI-S}{Section} \ref{S5555SI-S} to 
\hyperlink{S5757SID-W}{Section} \ref{S5757SID-W}. The main result is proved in \hyperlink{S5757SI-K}{Section} \ref{S5757SI-K} and \hyperlink{S5757GRP-V}{Section} \ref{S5757GRP-V}.

\section{Basic Theory of Smooth Numbers}\label{S5555SI-S}\hypertarget{S5555SI-S}
A few results in the theory of smooth numbers are stated in this section as a reference. The symbol $P(n)$ represents the largest prime divisor of an integer $n\geq1$ with the convention $P(1)=1$. For $x>y>1$, the parameter $u=\log x/\log y$. For $1<y<x$, an integer $n\geq1$ is called $y$-smooth if all its prime factors $p\mid n$ satisfy $p\leq y$. Define the set of smooth integers by 
\begin{equation}
	S(x,y)=\{n\leq x: P(n)\leq y\}
\end{equation} and 
its complementary set
\begin{equation}
	\overline{S}(x,y)=\{n\leq x: P(n)> y\}.
\end{equation}
The corresponding counting functions are defined by
\begin{equation}
	\Psi(x,y)=\#\{n\leq x: P(n)\leq y\}
\end{equation} and 
its complementary set
\begin{equation}
	\Phi(x,y)=\#\{n\leq x: P(n)> y\}.
\end{equation}
 
\begin{dfn}\label{dfn5555SI.200D}\hypertarget{dfn5555SI.200D}{\normalfont The Dickman function $\rho:[0,\infty]\longrightarrow [0,1]$ is a unique continuous function such that
\begin{equation}
	\rho(u)=
	\begin{cases}
		1&\text{ if } 0\leq u\leq 1,\\	
		1-\log u&\text{ if } 1\leq u\leq 2
	\end{cases}
\end{equation}	
and satisfies the differential equation
\begin{equation}
u\rho^{\prime}(u)-\rho(u-1)=0,
\end{equation}
for $u\geq1$. 	 
	}
\end{dfn}
As usual, the parameter $u$ is defined by
\begin{equation}
	u=\frac{\log x}{\log y},
\end{equation}
and the cardinalities of the sets $S(x,y)$ and its complement $\overline{S}(x,y)$ have the asymptotics 
\begin{equation}
\Psi(x,y)=\Psi(x,x^{1/u})\sim x\rho(	u),
\end{equation} 
and \begin{equation}
	\Phi(x,y)=\Phi(x,x^{1/u})\sim x(1-\rho(	u)).
\end{equation}
Other properties of interest in this analysis are listed here. 
\begin{thm}\label{thm5555SI.200T}\hypertarget{thm5555SI.200T}The function $\rho(u)$ satisfies the folowing.
\begin{enumerate}[font=\normalfont, label=(\roman*)]\label{eq5555SI.100S10}
	\item$\displaystyle \rho(u)>0$,\tabto{9cm}for $u>0$,	
	\item$\displaystyle u\rho(u)=\int_{u-1}^u\rho(t)dt $,\tabto{9cm}for $u\geq 1$,
	\item$\displaystyle \rho(u)\leq \frac{1}{\Gamma(u+1)} $,\tabto{9cm}for $u\geq 0$,
	\item$\displaystyle \rho(u)=e^{-u\left(\log u +\log\log u -1 +o(1)\right) }$, \tabto{9cm}for $u\geq 1$,
	\item$\displaystyle \rho^{\prime}(u)<0$,\tabto{9cm}for $u>1$,	
\end{enumerate}
\end{thm}
An elementary introduction to smooth numbers appears in {\color{blue}\cite[Chapter 9]{DL2012}}.

The smooth numbers counting function has the asymptotic formula

\begin{thm}\label{thm5555SI.400D}\hypertarget{thm5555SI.400D} For any fixed $\varepsilon>0$ the relation
	\begin{equation}
		\Psi(x,y)=x\rho(u)\left(1+O\left( \frac{\log(u+1)}{\log y}\right)  \right) 
	\end{equation} holds uniformly (with $u=\log x/\log y$) in
	the range
	\begin{equation}
		y\geq2 \quad \text{ and }\quad 1\leq u\leq (\log y)^{3/5} .
	\end{equation}
\end{thm}
\begin{proof}[\textbf{Proof}] The result is {\color{blue}\cite[Theorem 1.1]{HA1993}}, a slightly different form of this asymptotic formula appears in {\color{blue} \cite[Theorem 9.14]{DL2012}}.
\end{proof}

Surprisingly, this result proved in 1993 seems to imply the latest results for short intervals discussed in \hyperlink{S5757SI-A}{Section} \ref{S5757SI-A}.

\section{Smooth Integers in Short Intervals}\label{S5757SI-A}\hypertarget{S5757SI-A}
The topic of smooth integers in short intervals is a major area of research in the theory of smooth integers. 


\begin{thm}\label{thm5757SI.450}\hypertarget{thm5757SI.450} Let $17/30 < \theta \leq 1$. There exists a constant $C = C(\theta) > 0$ such that the estimate
\begin{equation}
\frac{\Psi(x+h,y)-\Psi(x,y)}{h} =\frac{\Psi(x,y)}{x}\left( 1+O_{\theta}\left(\frac{\log(u+1) }{\log y} \right) \right) 
\end{equation}
holds uniformly for

\begin{equation}
	x^{\theta}\leq h\leq x\quad \text{ and }\quad e^{C(\log x)^{2/3}(\log\log x)^{4/3}} \leq y \leq 2x.
\end{equation}
\end{thm}
\begin{proof}[\textbf{Proof}] The result is {\color{blue}\cite[Theorem 1.1]{YK2024}}.
\end{proof}

Another result for shorter intervals is stated here.

\begin{thm}\label{thm5757SI.400}\hypertarget{thm5757SI.400} For any $ \varepsilon > 0$, there exists a positive constant $ c_0 = c_0(\varepsilon)$ such that the following holds. If $x$ is large in terms of $\varepsilon$,
\begin{equation}
	\exp\left( c_0 (\log x)^{2/3} (\log \log x)^{4/3} \right) \leq y \leq x^{\frac{1}{c_0}},
\end{equation}
	and
\begin{equation}
	h \geq \sqrt{x} \exp\left( (1 + \epsilon) \left( \frac{11}{16}u \log u + 2 \log \log x \right) \right),
\end{equation}
where $u = \frac{\log x}{\log y} $, then the interval $ [x, x+h] $ contains a $ y$-smooth number.
\end{thm}

\begin{proof}[\textbf{Proof}] The result is {\color{blue}\cite[Theorem 1.2]{SJ2025}}.
\end{proof}
Stronger results, which extend both the smoothness ranges and the interval lengths, but conditional on RH are also available.

\begin{cor}\label{cor5757SI.350A}\hypertarget{cor5757SI.350A} Let $\theta=3/5$ and let $x>1$ be a sufficiently large real number. Set the parameters
\begin{align}
y&=e^{(\log x)(7\log\log x)^{-1}}\\
h&=x^{3/5}.
\end{align} 
Then the short interval
	\begin{equation}\label{eq5757SI.350A02}
	(\;	x, \;	x+h\;]
	\end{equation}
	contains 
	\begin{equation}\label{eq5757SI.350A04}
\Psi(x+h,y)-\Psi(x,y)=x^{3/5}\rho(u)(1+o(1)) 
\end{equation}	
$y$-smooth integers.
\end{cor}
\begin{proof}[\textbf{Proof}] The parameters 
\begin{fleqn}[0pt]	
	\begin{align}\label{eq5757SI.350A06}
		y&=e^{(\log x)(7\log\log x)^{-1}}>	\exp\left( c_0 (\log x)^{2/3} (\log \log x)^{4/3} \right),\\[.3cm]	
		h&=x^{3/5}> x^{7/12},\\[.3cm]	
		u&=\frac{\log x}{\log y}=\frac{\log x}{(\log x) (7\log \log x)^{-1}}=7\log \log x,\\[.3cm]	
		\rho(u)&=e^{-u\left(\log u +\log\log u -1 +o(1)\right) },
	\end{align}
\end{fleqn}	
lie in the admissible range of \hyperlink{thm5757SI.450}{Theorem} \ref{thm5757SI.450} and \hyperlink{thm5757SI.400}{Theorem} \ref{thm5757SI.400}. The properties of Dickman function $\rho(u)$ are described in \hyperlink{S5555SI-S}{Section} \ref{S5555SI-S}. Thus, the short interval
\begin{equation}\label{eq5757SI.350A20}
(x,x+h]=\left[\;	x, \;	x+x^{3/5}\;\right]
\end{equation}
contains
\begin{eqnarray}\label{eq5757SI.350A22}
	\frac{\Psi(x+h,y)-\Psi(x,y)}{h} &=&\frac{\Psi(x,y)}{x}\left( 1+O_{\theta}\left(\frac{\log(u+1) }{\log y} \right) \right) \nonumber\\[.2cm]
&=&\rho(u)\left( 1+O_{\theta}\left(\frac{\log\log\log x }{\log \log x} \right) \right)
\end{eqnarray}
$y$-smooth integers. The exponentially large short interval $[x,x+x^{3/5}]$ contains a substantial number of $y$-smooth integers. Specifically, \eqref{eq5757SI.350A22} blueuces to
\begin{eqnarray}\label{eq5757SI.350A28}
	\Psi(x+h,y)-\Psi(x,y)&=&x^{3/5}\rho(u)\left( 1+O_{\theta}\left(\frac{\log\log\log x }{\log \log x} \right) \right) \nonumber\\[.2cm]
	& =&x^{3/5}\rho(u)(1+o(1)).
\end{eqnarray}	 
This completes the proof.
\end{proof}

\begin{lem}\label{lem5757SI.335A}\hypertarget{lem5757SI.335A} Let $\theta =3/5$ and let $x>1$ be a large real number. Set
	\begin{equation}
		 y=e^{	(\log x) (7\log \log x)^{-1}} .
	\end{equation}
Assume that \begin{equation}\label{eq1001.335A00}
	N\in(\;	x, \;	x+x^{3/5}\;].
\end{equation}
and 
\begin{equation}\label{eq1001.335A01}
	\omega(N)\leq (\log \log x) +3(\log \log x)^{1/2}.
\end{equation} 
Then there are sufficiently many $y$-smooth integers $N\in [x,x+h]$ such that the radicals satisfy the inequality
\begin{equation}\label{eq1001.335A02}
\rad(N)=\prod_{p\mid N}p\leq 	e^{(\log x)(1+o(1))/7}.
\end{equation} 
\end{lem}
\begin{proof}[\textbf{Proof}] By hypothesis, the number of prime divisors 
	\begin{equation}\label{eq1001.335A04}
p_1<p_2<\;\cdots \;<p_w
\end{equation}	
	of the $y$-smooth integer $N$ is  
	\begin{equation}\label{eq1001.335A06}
w=\omega(N)= (\log \log x) +3(\log \log x)^{1/2}.
	\end{equation}
The radical of the $y$-smooth integer $N$ satisfies
\begin{eqnarray}\label{eq1001.335A08}
	\rad(N) &=&\prod_{p\mid N}p\nonumber\\[.2cm]
	&\leq & \left( e^{	 (\log x) (7\log \log x)^{-1}}\right)^{\log\log  x+3(\log \log x)^{1/2}}  \nonumber\\[.2cm]
	&\leq &	e^{(\log x)(1+o(1))/7}.
\end{eqnarray}
This completes the verification.
\end{proof}

\section{Distribution of the Omega Functions and Smooth Integers}\label{S5757SID-W}\hypertarget{S5757SID-W}
The standard normal distribution is defined by the integral 
\begin{equation}
	\Phi(z)=\frac{1}{\sqrt{2\pi}}\int_{-\infty}^ze^{-\frac{t^2}{2}}\;dt.
\end{equation}
The additive function $\omega(n)=\#\{p:p\mid n\} $ tallies the number of distinct prime divisors of $n$. Over the long interval $[1,x]\subset \N=\{1,2,3,\ldots\}$, the Erdos-Kac theorem asserts that $\omega(n)$ has a normal distribution with mean $\mu=\log\log x$ and variance $\sigma^2=\log\log x$. Various restrictions to various subsets  $S(x,y)]\subset \N$ of $y$-smooth integers were proved in \cite{AK1982}, \cite{HD1984}, \cite{DE2018}, \cite{MM2007}, et alii. These results provide an Erdos-Kac type theorems for the random variables $\omega(n)$ and $\Omega(n)$ restricted to the sets 
\begin{equation}
	S(x,y)
\end{equation} and in a suitable range of the parameters $x$ and $y$. The exact statement of the theorem is the following. 
\begin{thm} \label{thm5757WW.666}\hypertarget{thm5757WW.666} For any real number $x$ and $z=O(1)$, the number of prime divisors $\omega(n)$ of an integer $n\in[1,x]$ satisfies
\begin{equation}
	\frac{1}{x}\cdot \#\left\{n \in \leq x :\frac{ \omega(n)-\log\log x}{\sqrt{\log\log x}}\leq z\right \}\longrightarrow \Phi(z)
\end{equation}
as $y\to\infty$	holds in the range $u=o(\log\log y)$.
\end{thm}
\begin{proof}[\textbf{Proof}] The result is due to \cite{EK1940}, many different versions and related proofs are established in the literature, \cite{BT1969}, et alii.
\end{proof}

\subsection{Distribution for Small Parameters}
As a convenience reference, the parameters of interest for the current application to the abc conjecture are repeated here.
\begin{fleqn}[0pt]	
	\begin{align}\label{eq5757SI.700D06}
		y&=e^{	 (\log x) (7\log \log x)^{-1}}>	\exp\left( c_0 (\log x)^{2/3} (\log \log x)^{4/3} \right),\\[.3cm]	
		\log y&= (\log x) (7\log \log x)^{-1},\\[.3cm]	
		h&=x^{3/5}> x^{7/12},\\[.3cm]	
		u&=\frac{\log x}{\log y}=	(\log \log x)(1+o(1)),\\[.3cm]	
		\log u&=(\log\log\log x)(1+o(1)) ,\\[.3cm]	
		\log\log u&=(\log\log\log\log x)(1+o(1)),\\[.3cm]	
	\end{align}
\end{fleqn}	
The above choice of parameters simplifies the proof of \hyperlink{cor5757WW.750J}{Corollary} \ref{cor5757WW.750J} for the density of smooth integers in short intervals. 
\begin{thm} \label{thm5757WW.600}\hypertarget{thm5757WW.600} For any real number $z$ and $u=o(\log\log y)$, the number of prime divisors of smooth integer satisfies
	\begin{equation}
		\frac{1}{\Psi(x,y)}\cdot \#\left\{n \in S(x,y) :\frac{ \omega(n)-\log\log y}{\sqrt{\log\log y}}\leq z\right \}\longrightarrow \Phi(z)
	\end{equation}
	as $y\to\infty$	holds in the range $u=o(\log\log y)$.
\end{thm}
\begin{proof}[\textbf{Proof}] The result is {\color{blue}\cite[Theorem 1.1]{MM2007}}.
\end{proof}

For $20\leq h\leq X$, $x\in [X,2X]$ and $ y \in \R$, define the subset of integers
\begin{equation}\label{eq5757WW.610A00}
	F_{[x,x+h]}(z)=\{x<n\leq x+h:\omega(n)\leq \log_2 X+z\sqrt{\log_2 X}\}
\end{equation}
and
\begin{equation}\label{eq5757WW.610A02}
	\Phi_{X}(z)=\frac{1}{\sqrt{2\pi}}\int_{-\infty}^{z}e^{-\frac{t^2}{2}}\;dt +\frac{e^{-z^2/2}}{2\pi\sqrt{\log_2 X}}\left(\frac{2}{3}-c_1+\frac{z^2}{6}-\left(\log_2 X+z\sqrt{\log_2 X} \right)  \right) .
\end{equation}

\begin{thm} \label{thm5757WW.610}\hypertarget{thm5757WW.610} Uniformly for $20\leq h\leq X$, $x\in [X,2X]$ and $ \alpha\geq1$, on a
	\begin{equation}\label{eq5757WW.610A04}
		\Bigg |\Bigg |F_{[x,x+h]}(z)-\Psi_{X}(z)\Bigg |\Bigg|_{\infty}\ll\frac{\log_3 X}{\log_2 X}+\frac{\alpha^2(\log_2 h)^2}{\log h}
	\end{equation}
	for all $x\in [X,2X]$, except on a set $\mathcal{E}_X$ of measure at most 
	\begin{equation}\label{eq5757WW.610A06}
		\#\mathcal{E}_X=X\left( \frac{1}{(\log h)^{\alpha}}+\frac{1}{(\log X)^{1/150}}  \right) .
	\end{equation}
\end{thm}
\begin{proof}[\textbf{Proof}] The result is {\color{blue}\cite[Theorem 1]{DE2018}}.
\end{proof}

\begin{thm} \label{thm5757WW.620}\hypertarget{thm5757WW.620} For any real number $\lambda >0$ let
	\begin{equation}
		\Psi_{\lambda}(x,y)=\sum_{\substack{n\in \overline{S(x,y)} \\ \omega(n)-\log \alpha<\lambda\sqrt{\log \alpha}}}1
	\end{equation}
	and let $\Phi(x,y)=\#\overline{S}(x,y)$. Then, for $2\leq y<x$	
	\begin{equation}
		\Bigg |\frac{\Psi_{\lambda}(x,y)}{\Psi(x,y)}-\frac{1}{\sqrt{2\pi}}\int_{-\infty}^{\lambda}e^{-\frac{t^2}{2}}\;dt\Bigg |\ll\frac{1}{\log\alpha}.
	\end{equation}
	The implied constant is absolute and $\alpha=\log x/\log\log x$.
\end{thm}
\begin{proof}[\textbf{Proof}] The result is {\color{blue}\cite[Theorem A]{AK1982}}.
\end{proof}

\subsection{Distributions for Large Parameters}
The analog result for the range 
\begin{equation}\label{eq5757WW.700J00}
	2\log\log x +1\leq\log y\leq (\log x)^{3/4}
\end{equation}
were developed in \cite{HD1984}, and in other work. Equivalently
\begin{equation}\label{eq5757WW.700J02}
	(\log x)^{1/4}	\leq u\leq \frac{\log x}{2\log\log x+1}
\end{equation}
The corresponding means and variances for the omegas functions are these:
\begin{table}[H]
	\setlength{\tabcolsep}{.340513cm}
	\renewcommand{\arraystretch}{1.0950}
	\setlength{\arrayrulewidth}{.97pt}
	\centering
	\begin{tabular}{|c|l|l|l|}
		\hline
		Function & Mean & Variance&Range \\
		\hline
		$\omega$ & $\mu_{\omega}=u+O\left(\frac{u}{\log y} \right) $ & $\sigma_{\omega}^2=u$&$|k-k_{\omega}|\leq u^{1/2+\delta}$ \\
		\hline
		$\Omega$	& $\mu_{\Omega}=u+\frac{u}{\log u}+O\left(\frac{u}{(\log u)^2} \right) $ & $\sigma_{\Omega}^2=\frac{u}{(\log u)^2}$&$|k-k_{\omega}|\leq u^{1/14}$ \\
		\hline
	\end{tabular}
	\caption{Distribution Statistics for the Omega Functions}
	\label{table-700}
\end{table}
Recall the definition of the subsets
Let 
\begin{equation}
	\Psi(x,y)=\#\left\{1\leq n \leq x :P(n)\leq y\right \}
\end{equation}
and 
\begin{equation}
	\Psi_k(x,y)=\#\left\{1\leq n \leq x :P(n)\leq y \;\text{ and }\;\omega(n)=k\right \}.
\end{equation}

\begin{thm} \label{thm5757WW.700J}\hypertarget{thm5757WW.700J} Let $x>1$ be a large real number. Suppose that
	\begin{enumerate}[font=\normalfont, label=(\roman*)]
		\item $\displaystyle u\in [(\log x)^{1/4}, (\log x)/(\log\log x)]$, \tabto{6cm},	
		\item $\displaystyle |k-\mu_{\omega}|\leq u^{1/2-\delta}$, \tabto{9cm}where $\delta>0$.
	\end{enumerate}	
	Then, for any real number $z$, the number of prime divisors of smooth integer satisfies the followings.
	\begin{enumerate}[font=\normalfont, label=(\roman*)]
		\item \begin{eqnarray}
			\Psi_k(x,y)	&=&\Psi(x,y)\cdot \frac{1}{\sqrt{2\pi \sigma_{\omega}^2}}e^{-\frac{(k-\mu_{\omega})^2}{2\sigma_{\omega}^2}}\left(1+O\left(\frac{1}{\log\log x} \right)  \right) \nonumber,
		\end{eqnarray} 
		\item \begin{eqnarray}
			\sum_{1\leq k\leq z}	\Psi_k(x,y)
			&=&\Psi(x,y)\cdot \frac{1}{\sqrt{2\pi \sigma_{\omega}^2}}\int_{-\infty}^ze^{-\frac{(t-\mu_{\omega})^2}{2\sigma_{\omega}^2}}\,dt\left(1+O\left(\frac{1}{\log\log x} \right)  \right) \nonumber,
		\end{eqnarray} 
	\end{enumerate}	
	as $x\to\infty$	holds.
\end{thm}
\begin{proof}[\textbf{Proof}] The result is {\color{blue}\cite[Theorem 1]{HD1984}}, a similar result is also proved in \cite{HA1987}.
\end{proof}

\subsection{Distributions for Medium Parameters}
A result for the midsize range of the parameter $u=O(\log\log x)$ is developed in this subsection. The main goal is obtaining effective estimates of the mean and variance of the random variable $\omega(n)$ and to establish a normal distribution for $y$-smooth integers over the short interval $(x,x+h]$. In general, the mean $\mu_{\omega}$ and variance $\sigma_{\omega}^2$ of the random variable $\omega(n)$ of $n\in (x,x+h]$ are defined by
\begin{equation}
	\mu_\omega =\frac{1}{\Psi(x,y,h)}\sum_{\substack{n\in [x,x+h]\\P(n)\leq y} }\omega(n)
\end{equation}

and

\begin{equation}
	\sigma_\omega^2 =\frac{1}{\Psi(x,y,h)}\sum_{\substack{n\in [x,x+h]\\P(n)\leq y} }\left( \omega(n)-	\mu_\omega\right) ^2.
\end{equation}
\begin{thm} \label{thm5757WW.777E}\hypertarget{thm5757WW.777E} Let $X$ be a normal random variable with mean $\mu>0$ and variance $\sigma^2=\mu>0$ on the probability space $\Omega$. Suppose that the restriction of the random variable $X$ to the subset $\Omega_0\subset \Omega$ has the mean and variance 
	\begin{enumerate}[font=\normalfont, label=(\roman*)]\label{eq5757WW.777E00}
		\item $\displaystyle  \mu_0=\mu(1+o(1))$, \tabto{6cm},	
		\item $\displaystyle\sigma_0^2=\sigma^2(1+o(1))$,
		\item $\displaystyle p(x|\Omega_0)=p(x)(1+o(1))$
		\item $\displaystyle |X-\mu_0|\ll\mu$
	\end{enumerate}	
	as $\mu\to \infty$.	Then, the restriction of the random variable $X$ to the subset $\Omega_0$ is a normal random variable.
\end{thm}
\begin{proof}[\textbf{Proof}] By hypothesis $X$ is a normal random variable, and the standardized random variable is
	\begin{equation}\label{eq5757WW.777E02}
		Z=	\frac{X-\mu}{\sigma}.
	\end{equation}
	The corresponding characteristic function is 
	\begin{align}\label{eq5757WW.777E04}
		\phi_{Z}(t)&=E(e^{it\cdot \frac{X-\mu}{\sigma}})\\[.2cm]&=\int_{-\infty}^{\infty} p(x)e^{it\cdot \frac{X-\mu}{\sigma}}dx \nonumber\\[.2cm]
		&=e^{ \frac{-it\mu}{\sigma}}\sum _{n\geq0}\frac{1}{n!}\left( \frac{it}{\sigma}\right) ^nE[X^{n}]\nonumber.
	\end{align}
	where $p(x)$ is the probability function of $X$. Now consider the standardized restricted random variable
	\begin{equation}\label{eq5757WW.777E06}
		Z_0=\frac{X-\mu_0}{\sigma_0}=\frac{X-\mu(1+o(1))}{\sigma\sqrt{(1+o(1))}}
	\end{equation}	
	and its characteristic function 
	\begin{align}\label{eq5757WW.777E08}
		\phi _{Z_0}(t)&=E\left(e^{it\cdot \left(\frac{X-\mu(1+o(1))}{\sigma\sqrt{(1+o(1))}}\right)}\right)\\[.2cm]
		&=\int_{-\infty}^{\infty} p(x|\Omega_0)e^{it\cdot \left(\frac{X-\mu(1+o(1))}{\sigma\sqrt{(1+o(1))}}\right)}dx \nonumber\\[.2cm]
		&=e^{-it\cdot \frac{\mu\sqrt{(1+o(1))}}{\sigma}}\sum _{n\geq0}\frac{1}{n!}\left( \frac{it}{\sigma\sqrt{(1+o(1))}}\right) ^nE_0[X^{n}]\nonumber.
	\end{align}
	Moreover, using (iv) leads to
	\begin{align}\label{eq5757WW.777E10}
		Z_0&=\frac{X-\mu_0}{\sigma_0}\\[.2cm]
		&=\frac{X-\mu(1+o(1))}{\sigma\sqrt{(1+o(1))}}\nonumber\\[.2cm]
		&=\frac{X-\mu}{\sigma}-o\left( \frac{X-\mu(1+o(1))}{\sigma}\right) \nonumber\\[.2cm]
		&=\frac{X-\mu}{\sigma}-o\left( \sqrt{\mu}\right) \nonumber\\[.2cm]
		&=\frac{X-\mu}{\sigma}+O\left( \frac{\log\log\log x}{\sqrt{\log\log x}}\right)\nonumber,
	\end{align}	
	where $o(1)=(\log\log\log x)/(\log\log x)$. Since 
	\begin{align}\label{eq5757WW.777E12}
		p(x|\Omega_0)=e^{-\frac{(X-\mu_0)^2}{2\sigma_0^2}}&=e^{-\frac{(X-\mu)^2}{2\sigma^2}}+O\left( \frac{\log\log\log x}{\sqrt{\log\log x}}\right) ,
	\end{align}
	the conditional probability function $p(x|\Omega_0)$ converges to $p(x)=e^{-\frac{(X-\mu)^2}{2\sigma^2}}$ as $x$ tends to infinity, this is like an application of Slutsky's theorem. In addition, the limit
	
	\begin{equation}\label{eq5757WW.777E14}
		\lim_{x\to\infty}\;	\phi _{Z_0}(t)=\lim_{x\to\infty}\;e^{-it\cdot \frac{\mu\sqrt{(1+o(1))}}{\sigma}}\sum _{n\geq0}\frac{1}{n!}\left( \frac{it}{\sigma\sqrt{(1+o(1))}}\right) ^nE_0[X^{n}]=	\phi_{Z}(t).
	\end{equation}
	holds, this is like an application of Levy's continuity theorem. This proves that under the stated hypothesis in \eqref{eq5757WW.777E00}, the restriction of the random variable $X$ to the subset $\Omega_0\subset \Omega$.
\end{proof}

Let 
\begin{eqnarray}
	\Psi(x,y,h)&=&\Psi(x+h,y)-\Psi(x,y)\\[.2cm]&=&\#\left\{x\leq n \leq x+h :P(n)\leq y\right \} \nonumber
\end{eqnarray}
and 
\begin{eqnarray}
	\Psi_k(x,y,h)&=&\Psi_k(x+h,y)-\Psi_k(x,y)\\[.2cm]
	&=&\#\left\{x\leq n \leq x+h :P(n)\leq y\; \text{ and }\;\omega(n)=k\right \}\nonumber.
\end{eqnarray}
\begin{thm} \label{thm5757WW.720J}\hypertarget{thm5757WW.720J} Let $x>1$ be large and consider the short interval $(x,x+h]$. Suppose that
	\begin{enumerate}[font=\normalfont, label=(\roman*)] \label{eq5757WW.720J00}
		\item$\displaystyle y=e^{(\log x)(7\log \log x)^{-1}}>e^{	 c_0 (\log x)^{2/3} (\log \log x)^{4/3}}$, \tabto{9cm}for some constant $c_0>0$,
		\item $\displaystyle h=x^{3/5}>x^{7/12}$.
	\end{enumerate}	
	Then, uniformly for any real number $z=O(\log y)$, the number of distinct prime divisors of $y$-smooth integer satisfies the followings.

	\begin{eqnarray} \label{eq5757WW.720J02A}
		\Psi_k(x,y,h)&=&\Psi(x,y,h)\cdot \frac{1}{\sqrt{2\pi \sigma_{\omega}^2}}e^{-\frac{(k-\mu_{\omega})^2}{2\sigma_{\omega}^2}}\left(1+O\left(\frac{1}{\log\log x} \right)  \right) ,
	\end{eqnarray} 
	and 
\begin{eqnarray}\label{eq5757WW.720J02B}
\sum_{1\leq k\leq z}	\Psi_k(x,y,h)
&=&\Psi(x,y,h)\cdot \frac{1}{\sqrt{2\pi \sigma_{\omega}^2}}\int_{-\infty}^ze^{-\frac{(t-\mu_{\omega})^2}{2\sigma_{\omega}^2}}\,dt\left(1+O\left(\frac{1}{\log\log x} \right)  \right) ,
\end{eqnarray}
	where $ \mu_{\omega}=(\log\log x)(1+o(1))$ and $\sigma_{\omega}^2(\log\log x)(1+o(1))$ are the mean and variance of the random variable $\omega(n)$ of $y$-smooth integers $n\in(x,x+h]$,
	as $x\to\infty$ hold.
\end{thm}  
 
\begin{proof}[\textbf{Proof}] Sketch of the lengthy proof for the $y$-smooth parameters stated in \eqref{eq5757SI.700D06}. The given small $u=7\log\log x$ grows slowly while $y$ grows significantly faster. This condition provides control over the various asymptotic formulas for smooth-numbers in short intervals. The large interval size $h = x^{3/5}>x^{7/12}$ guarantees that the asymptotic formulas are unconditional.  Partial details of the saddle-point analysis are as follows below. 
Let $S(y)=\{n\in\N:P(n)\leq y\}$ be the set of $y$-smooth integers and let 
\begin{equation}
\zeta(s,y)^z=\sum_{n\in S(y)}\frac{z^{\omega(n)}}{n^s} = \prod_{p\leq y}\left( 1-\frac{1}{p^s}\right) ^{-z}.
\end{equation}
be the corresponding zeta function. The complex function 
\begin{equation}
	\Psi(x,y,h;z)\ :=\ \sum_{\substack{n\in (x,x+h]\\P(n)\le y}}z^{\omega(n)}
\end{equation}
of the complex number in the small disk $z\in D=\{|z-1|\ll(\log\log y)^{-1/2+\delta}\}$ is recoveblue by the truncated Perron formula
\begin{equation}
	\Psi(x,y,h;z) = \frac1{i2\pi }\int_{c-iT}^{c+iT}\zeta(s,y)^z\left( \frac{(x+h)^s}{s}- \frac{x^s}{s}\right) \,ds\ +\ O\left( \frac{x\log x}{T}\right),
\end{equation}
where $c>1$ is a convenient abscissa and $T$ is a large height, and $z=e^{it}$ with $|t|\ll(\log\log y)^{-1/2+\delta}$, see {\color{blue}\cite[Section 7.1]{YK2024}}. \\

Under these conditions, the mean $\mu_{\omega}$ and variance $\sigma_{\omega}^2$ of the random variable $\omega(n)$ over the short interval $(x,x+h]$ are approximately the same as the quantities for random variable over the long interval $ [1,x]$ in  \hyperlink{thm5757WW.666}{Theorem} \ref{thm5757WW.666}. These are
\begin{equation}
	\mu_\omega =\log\log y= (\log\log x)\left( 1+O\left( \frac{\log\log\log x}{\log\log x}\right) \right) 
\end{equation}

and

\begin{equation}
	\sigma_\omega =\sqrt{\log\log y}= \sqrt{(\log\log x)\left( 1+O\left( \frac{\log\log\log x}{\log\log x}\right) \right)},
\end{equation} 
see {\color{blue}\cite[Section 2]{MM2007}}.\\

An application of \hyperlink{thm5757WW.777E}{Theorem} \ref{thm5757WW.777E} to the random variable $X=\omega(n)$ leads to the local central limit theorem
\begin{equation}
	\Psi_k(x,y,h) = \Psi(x,y,h)\cdot\frac1{\sqrt{2\pi\sigma_\omega^2}}e^{-\frac{(k-\mu_\omega)^2}{2\sigma_\omega^2}}\left( 1+O\left( \frac{\log\log\log x}{\sqrt{\log\log x}}\right) \right) 
\end{equation}
uniformly for $k$ in the range $|k-\mu_\omega|\ll\log\log x$. Summing the local probabilities for $k\leq z$ yields the cumulative form in \eqref{eq5757WW.720J02B}. Similar results for different ranges of $u=\log x/\log y>0$ are stated in \hyperlink{thm5757WW.600}{Theorem} \ref{thm5757WW.600} to \hyperlink{thm5757WW.700J}{Theorem} \ref{thm5757WW.700J}. 
\end{proof}

\begin{cor}\label{cor5757WW.750J}\hypertarget{cor5757WW.750J} Let $\theta=3/5$ and let $x>1$ be a sufficiently large real number. Set the parameters
	\begin{align}
		y&=e^{(\log x)(7\log\log x)^{-1}}\nonumber\\[.2cm]
		h&=x^{3/5},\nonumber\\[.2cm]
		u&=7\log\log x,\nonumber\\[.2cm]
		z_0&=\mu_{\omega}+	s\sigma_{\omega}= (\log \log x) +3(\log \log x)^{1/2}.\nonumber
	\end{align} 
Then the short interval
	\begin{equation}\label{eq5757WW.750JA02}
		(\;	x, \;	x+h\;]
	\end{equation}
	contains 
	\begin{equation}\label{eq5757WW.750J04}
		\sum_{k\leq z_0}\left( \Psi_k(x+h,y)-\Psi_k(x,y)\right) \gg x^{3/5}\rho(u)(1+o(1)) 
	\end{equation}	
	$y$-smooth integers. In particular, it contains some $y$-smooth integers such that
\begin{equation}\label{eq5757WW.750J05}
		\omega(n)\ll(\log \log x) +3(\log \log x)^{1/2}.
\end{equation}
\end{cor}
\begin{proof}[\textbf{Proof}] To verify that half of the $y$-smooth integers $n\in[x,x+h]$ have at most $z_0=\omega(n)$ distinct prime factors, it will be shown that the distribution is normal and its complement subset is negligible. Given the range
\begin{equation}\label{eq5757WW.750J06}
	 u=7\log\log x=o((\log x)(7\log \log x)^{-1})
\end{equation}
the mean and standard deviation	 are
\begin{align}\label{eq5757WW.750J08}
\mu_{\omega}&=\log\log y=\log \log x\nonumber\\[.2cm]
	\sigma_{\omega}&= \sqrt{\log\log y}\approx(\log\log x)^{1/2}.
\end{align}
Let 
\begin{equation}\label{eq5757WW.750J10}
z_0=\mu_{\omega}+	s\sigma_{\omega}\approx\log\log x+3(\log\log x)^{1/2}.
\end{equation}
By \hyperlink{thm5757WW.720J}{Theorem} \ref{thm5757WW.720J} the cardinality of the subset
\begin{equation}\label{eq5757WW.750J12}
S_z(x,y,h)=\{x\leq n\leq x+h: P(n)\leq y\;\text{ and }\; \omega(n)>z_0 \}
\end{equation}
is 
\begin{eqnarray}\label{eq5757WW.750J14}
	\sum_{ k> z_0}	\Psi_k(x,y,h)
	&=&\Psi(x,y,h)\cdot \frac{1}{\sqrt{2\pi \sigma_{\omega}^2}}\int_{z_0}^{\infty}e^{-\frac{(t-\mu_{\omega})^2}{2\sigma_{\omega}^2}}\,dt\left(1+O\left(\frac{1}{\log\log x} \right)  \right) \nonumber\\[.2cm]
	&=&x^{3/5}\,\rho(u)\,\left(1- \Phi(z_0)\right) (1+o(1) ) \nonumber\\[.2cm]
	&=&(0.001349898)x^{3/5}\,\rho(u) (1+o(1) ). 
\end{eqnarray}	 
\end{proof}

\section{Exceptional Triples in Short Intervals}\label{S5757SI-K}\hypertarget{S5757SI-K}
This counterexample is based on the established theory of smooth integers in short intervals. The notations $\rad(abc)<c^{1-\varepsilon_0}$ and $c>c_{\varepsilon}\rad(abc)^{1+\varepsilon}$ are equivalent, but $\varepsilon_0=\varepsilon/(1+\varepsilon)$. So the proof can be written in either form. 
\begin{thm}\label{thm5757SI.425A}\hypertarget{thm5757SI.425A} Let $\theta=3/5>7/12$ and let $x>1$ be a sufficiently large real number. Suppose that 
\begin{enumerate}[font=\normalfont, label=(\roman*)]
\item$\displaystyle y=e^{	 (\log x) (7\log \log x)^{-1}}$, \tabto{8cm}Smooth parameter.
\item $\displaystyle h=x^{3/5}$, \tabto{8cm}Short interval size.
\end{enumerate}
Then, as $x\to\infty$, there are infinitely many $y$-smooth integers 
\begin{equation}\label{eq5757SI.425A02}
c\in (x,x+x^{3/5}],
\end{equation} 
such that $a+b=c$ and the abc inequality fails: 
\begin{equation}\label{eq5757SI.425A04}
	c\leq c_{\varepsilon}\left( \rad\left(abc \right)  \right)^{1+\varepsilon},
\end{equation}
where $\gcd(a,b,c)=1$.
\end{thm}
\begin{proof}[\textbf{Proof}] Fix a small number $\varepsilon>0$. Let $x=p^k$ be a large prime power, where $y<p$ and $k >4$ such that 
\begin{equation}\label{eq5757SI.425A06}
\left( \frac{3}{5}+\frac{1}{k}+\frac{1+o(1)}{7}\right)(1+\varepsilon)<1.
	\end{equation}
By \hyperlink{cor5757SI.350A}{Corollary} \ref{cor5757SI.350A} there exists a $y$-smooth integer in the short interval 
\begin{equation}\label{eq5757SI.425A08}
c=p_1^{e_1}\cdot p_2^{e_2}\;\cdots\;\cdot p_w^{e_w}\in (x,x+h ],
\end{equation}
where $p_i\leq y$ are distinct prime, $e_i\geq0$. Moreover, it follows from  \hyperlink{cor5757WW.750J}{Corollary} \ref{cor5757WW.750J}, that a positive density of $y$-smooth integers in $(x,x+h]$ have $\omega(n)\leq w_0$ prime factors, where
\begin{equation}
w_0\leq\log \log x+3(\log \log x)^{1/2}<2\mu_{\omega}.
\end{equation} 
Now, let 
\begin{align}\label{eq5757SI.425A10}
	a&=|p^k-p_1^{e_1}\cdot p_2^{e_2}\;\cdots\;\cdot p_w^{e_w}|\leq h\nonumber\\[.2cm]
	b&=p^k\nonumber\\[.2cm]
	c&=p_1^{e_1}\cdot p_2^{e_2}\;\cdots\;\cdot p_w^{e_w}.
\end{align}
A short inspection shows that, $\gcd(a,b,c)=1$, see \hyperlink{rmk5757GRP.002}{Remark} \ref{rmk5757GRP.002}. Next, to derive a contradiction, suppose that:
\begin{eqnarray}\label{eq5757SI.425A12}
c&\leq &c_{\varepsilon}\left( \text{rad}\left(abc \right)  \right)^{1+\varepsilon}\nonumber\\[.2cm]
&\leq &c_{\varepsilon}\left( \text{rad}\left(h\cdot p^k\cdot p_1^{e_1}\cdot p_2^{e_2}\;\cdots\;\cdot p_w^{e_w} \right) \right)^{1+\varepsilon}\nonumber\\[.2cm]
&\leq &c_{\varepsilon}\left( 	x^{\frac{3}{5}}\cdot  x^{\frac{1}{k}} \cdot  x^{\frac{1+o(1)}{7}}  \right)^{1+\varepsilon},
\end{eqnarray}
where the third line in \eqref{eq5757SI.425A12} follows from \hyperlink{lem5757SI.335A}{Lemma} \ref{lem5757SI.335A}. Rewrite it as
\begin{eqnarray}\label{eq5757SI.425A14}
c=x+O(x^{\frac{3}{5}})&\leq &c_{\varepsilon}\left( 	x^{\frac{3}{5}}\cdot  x^{\frac{1}{k}} \cdot  x^{\frac{1+o(1)}{7}}  \right)^{1+\varepsilon} \\[.3cm]
c=x(1+o(1))&\leq &c_{\varepsilon} \cdot  x^{\left( \frac{3}{5}+\frac{1}{k}+\frac{1+o(1)}{7}\right) (1+\varepsilon)}\nonumber,
\end{eqnarray}
set $k=5$. Taking logarithm yields
\begin{eqnarray}\label{eq5757SI.425A16}
\log x+\log(1+o(1))	&\leq &\log c_{\varepsilon} +(1+\varepsilon)\left( \frac{3}{5}+\frac{1}{k}+\frac{1+o(1)}{7}\right)\log x\nonumber\\[.3cm]
	1+\frac{\log(1+o(1))}{\log x}	&\leq &\frac{\log c_{\varepsilon}}{\log x} +(1+\varepsilon)\left( \frac{3}{5}+\frac{1}{5}+\frac{1+o(1)}{7}\right)\nonumber\\[.3cm]
	\frac{2-33\varepsilon+o(1)}{35}	&\leq &\frac{\log c_{\varepsilon}-\log(1+o(1)) }{\log x}.
\end{eqnarray}
Clearly, for $\varepsilon <2/33$ this is a contradiction infinitely often as $x\to \infty$. Specifically,
\begin{equation}\label{eq5757SI.425A18}
	\frac{c}{c_{\varepsilon}\left( \text{rad}\left(abc \right)  \right)^{1+\varepsilon}}
\end{equation}
is unbounded as $x\to \infty$.
\end{proof}
The determination of an exceptional triple of index $\gamma_{100}(abc)\approx 1.6393$ appears to involve calculations with extremely large integers $x=p^k=b$ to find a $y$-smooth integer $c$ in the interval $(x,x+x^{3/5}]$, this problem is quite similar to the Skewes number problem. \\

The cube $[1,x]\times[1,x]\times[1,x]\subset \N\times\N\times\N$ contains about $6\pi^{-2}x^2$ relatively primes triples $(a,b,c)$, the result in \cite{LJ2025} shows that about $\#\mathscr{E}(x)=O(x^{33/50})$ fail the $abc$ inequality. A proof of the $abc$ conjecture requires that just $\#\mathscr{E}(x)=O(1)$ fail the inequality. However, the analysis within seems to show that the cardinality of exceptional set $\mathscr{E}(x)$ is not bounded, specifically $\#\mathscr{E}(x)\gg x^{\delta}$ as $x\to\infty$ for some $\delta>0$.

\begin{rmk}\label{rmk5757GRP.002}\hypertarget{rmk5757GRP.002}{\normalfont
		In more detail, the relatively prime condition is as follows: Since $p>y\geq P(c)$, it follows that $p\nmid c$ and $\gcd(b,c)=1$. Any prime $q$ dividing both $a$ and $b$ divides $b=p^k$ and therefore $q=p$; but $p\nmid a$ because $a=|p^k-c|$ and $p\nmid c$. Thus $\gcd(a,b)=1$. Finally, if a prime $q$ divided $a,b,c$ simultaneously then $q\mid b$ and $q\mid c$, contradicting $p\nmid c$ and the fact that all prime divisors of $c$ are $\le y<p$. Therefore $\gcd(a,b,c)=1$.}
\end{rmk}

\section{A Lower Bound of the Exceptional Set}\label{S5757GRP-V}\hypertarget{S5757GRP-V}
The qualitative is much simpler and it is presented here.

\begin{proof}[{\color{blue}\textbf{Proof:} }{\normalfont \hyperlink{thm5757SI.940}{Theorem} \ref{thm5757SI.940}}] For each large real number $x>1$, \hyperlink{thm5757SI.425A}{Theorem} \ref{thm5757SI.425A} proves that exceptional set 
	\begin{equation}
		\mathscr{E}(x)\ne\varnothing
	\end{equation}
contains exceptional triple $(a,b,c)$ with $c>x$ that contradicts the abc inequality \eqref{eq5757SI.425A04}. Hence, for every sufficiently large $x>1$, there exists an exceptional triple with $c>x$. In particular, in the limit, the cardinality   	
	\begin{equation}
		\#	\mathscr{E}=\lim_{x\to\infty}\#	\mathscr{E}(x)
	\end{equation}
of the exceptional set is unbounded.
\end{proof}

For every sufficiently large X, there exists an exceptional triple with c>X.
?

Consequently the set of exceptional triples is infinite.
  

\end{document}